\documentclass[11pt]{amsart} 
\usepackage[foot]{amsaddr}
\usepackage{amsmath,amssymb,bm, color} 
 
\usepackage{amsmath,float}
\usepackage{amssymb}
\usepackage{graphicx}
 \usepackage{tikz}
  \numberwithin{dummy}{section}

\newtheorem{algorithm}{Least Squares Weak Galerkin Algorithm}

\newcommand{\bu}{{\bf u}}

\newcommand{\bw}{{\bf w}}

\newcommand{\bv}{{\bf v}}

\def\V{{\mathcal V}}
\def\W{{\mathcal W}}

\def\bbf{{\bf f}}
\def\bg{{\bf g}}

\def\3bar{{|\hspace{-.02in}|\hspace{-.02in}|}}

 \newtheorem{theorem}{Theorem}[section]
 \newtheorem{lemma}[theorem]{Lemma}

 \def\ad#1{\begin{aligned}#1\end{aligned}}  \def\b#1{{\bf #1}} 
\def\a#1{\begin{align*}#1\end{align*}} 
\def\an#1{\begin{align}#1\end{align}}

\def\p#1{\begin{pmatrix}#1\end{pmatrix}} 
  \numberwithin{equation}{section}
 
\def\boxit#1{\vbox{\hrule height1pt \hbox{\vrule width1pt\kern1pt
     #1\kern1pt\vrule width1pt}\hrule height1pt }}

\long\def\comment#1{}

\title[Least-Squares Weak Galerkin]
{Solving the Stokes Equations via a Least Squares Weak Galerkin Method}

  \author {Chunmei Wang$\dagger$}
  \address{Department of Mathematics, University of Florida, Gainesville, FL 32611, USA. }
  \email{chunmei.wang@ufl.edu}
 
\thanks{$\dagger$ \ Corresponding author. }

\author {Shangyou Zhang}
\address{Department of Mathematical Sciences,  University of Delaware, Newark, DE 19716, USA}   \email{szhang@udel.edu}  

\begin{document}
\begin{abstract}
We present a least-squares weak Galerkin (LS-WG) finite element method for solving the Stokes equations on arbitrary polygonal and polyhedral meshes. By utilizing discrete weak derivatives on discontinuous polynomial spaces, the proposed framework naturally accommodates complex domain geometries and general partitions. Crucially, this least-squares formulation bypasses the traditional inf-sup (LBB) compatibility condition, transforming the standard indefinite saddle-point problem into an inherently symmetric and positive definite (SPD) discrete linear system. We establish the well-posedness of the numerical scheme and rigorously derive optimal-order error estimates in a custom discrete energy norm. Specifically, we prove convergence rates of $\mathcal{O}(h^k)$ for the discrete projection error and $\mathcal{O}(h^{k-1})$ for the global approximation error when employing polynomials of degree $k \ge 1$ for the velocity field and $k-1$ for the pressure. Extensive numerical experiments confirm these theoretical convergence rates, demonstrating the method's robustness, geometric flexibility, and overall efficiency.
\end{abstract}


  
\maketitle
 
\section{Introduction}

The Stokes equations are fundamental to fluid mechanics, governing the behavior of incompressible viscous fluids in the creeping flow regime where the Reynolds number is very low. In this regime, viscous forces dominate inertial effects, making the model essential for simulating diverse phenomena across science and engineering. Key applications range from microfluidics and lab-on-a-chip technologies to biological systems—such as the swimming mechanisms of microorganisms—and large-scale geophysical processes, including mantle convection and glacier dynamics. Furthermore, the Stokes problem provides the underlying linear structure for the full, nonlinear Navier-Stokes equations. Consequently, the development of robust and geometrically flexible numerical solvers for the Stokes system remains a critical area of computational research.

In this paper, we introduce a least squares weak Galerkin (WG) finite element method for the Stokes equations. We seek a velocity field $\bu$ and a pressure $p$ that satisfy the following system:
\begin{align}
-\Delta\bu + \nabla p &= \bbf \quad \text{in } \Omega, \label{moment} \\
\nabla \cdot \bu &= 0 \quad \text{in } \Omega, \label{cont} \\
\bu &= \bg \quad \text{on } \partial\Omega, \label{bc}
\end{align}
where $\Omega \subset \mathbb{R}^d$ ($d=2,3$) is a polygonal or polyhedral domain, $\bbf$ is the body force, and $\bg$ represents the Dirichlet boundary data.

The numerical approximation of the Stokes problem presents significant challenges, primarily due to the indefinite saddle-point nature of the system. Standard mixed finite element methods (FEM) require that the velocity and pressure spaces satisfy the restrictive discrete inf-sup (Ladyzhenskaya-Babu\v{s}ka-Brezzi or LBB) condition to ensure stability and well-posedness \cite{ctvw2009}. Constructing inf-sup stable element pairs on complex or non-matching grids is often computationally intensive and geometrically limiting. 

To overcome these constraints, discontinuous Galerkin (DG) methods were developed to relax inter-element continuity, thereby enhancing geometric flexibility and enabling divergence-free velocity approximations \cite{cks2006}. More recently, the weak Galerkin (WG) finite element method has emerged as a powerful alternative.  The WG framework \cite{wg1, wg2, wg3, wg4, wg5, wg6, wg7, wg8, wg9, wg10, wg11, wg12, wg13, wg14, wg15, wg16, wg17, wg18, wg19, wg20, wg21, itera, wz2023, guan, guan2} utilizes discrete weak derivatives and enforces continuity across interfaces through specially designed stabilizers. First applied to the Stokes problem in \cite{wy2015}, the WG method naturally accommodates arbitrary polygonal and polyhedral meshes. Recent refinements have further introduced lifting operators and advanced stabilization to ensure pressure robustness across complex polytopal domains \cite{wwl2022, wz2025, my2016}.

A notable evolution within this field is the Primal-Dual Weak Galerkin (PDWG) method \cite{pdwg1, pdwg2, pdwg3, pdwg4, pdwg5, pdwg6, pdwg7, pdwg8, pdwg9, pdwg10, pdwg11, pdwg12, pdwg13, pdwg14, pdwg15}, which formulates numerical approximations as constrained minimization problems. By enforcing the governing equations as weak constraints through Lagrange multipliers, PDWG offers favorable stability and symmetry for various problems, including non-self-adjoint linear transport equations \cite{wwhyperbolic, pdwg1, pdwg6, pdwg12}.

Building upon these developments, this work proposes a novel least-squares weak Galerkin finite element method for the Stokes problem \eqref{moment}--\eqref{bc}. Unlike traditional WG formulations that may yield non-symmetric systems, or primal-dual WG (PDWG) methods that increase the total number of unknowns via dual variables, the proposed LS-WG method minimizes a least-squares functional in a weak sense. This approach offers several distinct advantages:

\begin{itemize}
    \item \textbf{Symmetry and Positivity:} The formulation transforms the standard indefinite saddle-point problem into a discrete linear system that is inherently symmetric and positive definite (SPD). This allows for the use of highly efficient iterative solvers, such as the conjugate gradient (CG) method.
    \item \textbf{Mesh Flexibility:} The method inherits the standard WG capability of handling arbitrary polytopal meshes with pendant nodes, eliminating the need for matching grids or specialized continuous basis functions.
    \item \textbf{No Inf-Sup Requirement:} The least-squares framework bypasses the need to satisfy the traditional LBB compatibility condition, greatly simplifying the choice of velocity and pressure finite element spaces.
\end{itemize}

This paper provides a rigorous theoretical and numerical foundation for the LS-WG scheme. We establish the existence and uniqueness of the discrete solution and derive optimal-order error estimates in a carefully defined discrete energy norm. Specifically, we prove convergence rates of $\mathcal{O}(h^k)$ for the discrete projection error and $\mathcal{O}(h^{k-1})$ for the global approximation error when employing polynomials of degree $k \ge 1$ for the velocity field and $k-1$ for the pressure. Finally, we present extensive numerical experiments to validate these theoretical convergence rates and demonstrate the robustness and geometric flexibility of the algorithm.

The remainder of this paper is organized as follows. In Section 2, we define the fundamental discrete weak differential operators: the discrete weak Laplacian, the discrete weak gradient, and the discrete weak divergence. Section 3 presents the development of the least-squares weak Galerkin finite element scheme for the Stokes equations. The error equations for the LS-WG approximation are derived in Section 4, followed by a rigorous derivation of optimal-order error estimates, specifically establishing the $\mathcal{O}(h^k)$ and $\mathcal{O}(h^{k-1})$ convergence rates in the discrete energy norm, in Section  5. Finally, Section 6 provides a series of numerical experiments to validate the theoretical findings and demonstrate the robustness of the method.

Let $D \subset \mathbb{R}^d$ ($d=2, 3$) be an open, bounded domain with a Lipschitz continuous boundary. We employ standard notation for Sobolev spaces. For any $s \ge 0$, $H^s(D)$ denotes the Sobolev space with the associated inner product $(\cdot,\cdot)_{s,D}$, norm $\|\cdot\|_{s,D}$, and seminorm $|\cdot|_{s,D}$. 
The space $H^0(D)$ coincides with $L^2(D)$, for which the inner product and norm are simplified to $(\cdot,\cdot)_{D}$ and $\|\cdot \|_{D}$, respectively. In cases where $D=\Omega$, the subscript $D$ will be omitted for brevity.

For vector-valued functions, we define the space $H(\text{div}; D)$ as the set of functions in $[L^2(D)]^d$ whose divergence also belongs to $L^2(D)$:
\begin{equation*}
H(\text{div}, D) = \left\{ \bv \in [L^2(D)]^d : \nabla \cdot \bv \in L^2(D) \right\}.
\end{equation*} 

\section{Discrete Weak Differential Operators}  

The hallmark of the weak Galerkin method is the replacement of classical (strong) derivatives with weak derivatives. These weak operators are defined on a class of weak functions that are discontinuous across element boundaries. For the least-squares WG formulation of the Stokes problem \eqref{moment}--\eqref{bc}, we require three specific weak differential  operators: the weak Laplacian, the weak divergence, and the weak gradient. For completeness, this section reviews the formal definitions of these weak  differential operators and their discrete counterparts \cite{wz2025, wy2015, wzholm}.

\subsection{Discrete Weak Laplacian}

Let $T$ be a polygonal or polyhedral domain in $\mathbb{R}^d$ ($d=2,3$) with boundary $\partial T$. We define a  {weak function} for the velocity field on $T$ as an ordered triplet $\bv=\{\bv_0, \bv_b, \mathbf{v}_g\}$, where the components are defined as follows:  $\bv_0 \in [L^2(T)]^d$ represents the value of $\bv$ in the interior of $T$;
  $\bv_b \in [L^2(\partial T)]^d$  represents the value of $\bv$ on the boundary $\partial T$;
   $\mathbf{v}_g \in [L^2(\partial T)]^{d\times d}$  represents the value of the gradient $\nabla \bv$ on the boundary $\partial T$.
Crucially, $\bv_b$ and $\mathbf{v}_g$ are defined independently and are not required to be the traces of $\bv_0$ or $\nabla \bv_0$ on $\partial T$. We denote the space of such weak functions on $T$ by:
\a{
\mathcal{W}(T) = \bigg\{ \bv= \{\bv_0, \bv_b, \mathbf{v}_g\}
  & : \bv_0 \in [L^2(T)]^d,\\
  &\quad  \bv_b \in [L^2(\partial T)]^d, \ \mathbf{v}_g \in [L^2(\partial T)]^{d\times d} \bigg\}. }
 
The  {weak Laplacian} of a function $\bv \in \mathcal{W}(T)$, denoted by $\Delta_w \bv$, is a linear functional in the dual space of $[H^2(T)]^d$ defined by:
\a{
(\Delta_w \bv, \bw)_T 
  &=(\bv_0, \Delta\bw)_T - \langle \bv_b, \nabla \bw \cdot \mathbf{n}\rangle_{\partial T}
  \\ & \quad \  + \langle \mathbf{v}_g \cdot \mathbf{n}, \bw\rangle_{\partial T}, \quad \forall \bw \in [H^2(T)]^d,
   }
where $\mathbf{n}$ is the outward unit normal vector on $\partial T$. 

Let $P_r(T)$ denote the space of polynomials of degree at most $r$ on $T$. The  {discrete weak Laplacian}, $\Delta_{w,r,T} \bv$, is the unique polynomial in $[P_r(T)]^d$ satisfying:
\an{\label{dislap}\ad{
(\Delta_{w,r,T} \bv, \bw)_T & = (\bv_0, \Delta \bw)_T - \langle \bv_b, \nabla \bw \cdot \mathbf{n} \rangle_{\partial T} \\
  &\quad \  + \langle \mathbf{v}_g \cdot \mathbf{n}, \bw \rangle_{\partial T},
    \quad \forall \bw \in [P_r(T)]^d.
 } }
If $\bv_0$ possesses sufficient regularity (e.g., $\bv_0 \in [H^2(T)]^d$), the discrete weak Laplacian can be equivalently expressed through integration by parts:
\an{ \label{dislap2} \ad{
(\Delta_{w,r,T} \bv, \bw)_T & = (\Delta \bv_0, \bw)_T + \langle \bv_0 - \bv_b, \nabla \bw \cdot \mathbf{n} \rangle_{\partial T} \\
  &\quad \ + \langle (\mathbf{v}_g - \nabla \bv_0) \cdot \mathbf{n}, \bw \rangle_{\partial T}, 
    \quad \forall \bw \in [P_r(T)]^d. } }

\subsection{Discrete Weak Divergence}

For any $\bv \in \mathcal{W}(T)$, we define the weak divergence $\nabla_w \cdot \bv$ as a linear functional in the dual space of $H^1(T)$ such that for each $\varphi \in H^1(T)$:
\begin{equation*}\label{wd}
(\nabla_w \cdot \bv, \varphi)_T = -(\bv_0, \nabla\varphi)_T + \langle \bv_b \cdot \mathbf{n}, \varphi\rangle_{\partial T}.
\end{equation*}
The  {discrete weak divergence} operator $\nabla_{w,r,T} \cdot \bv \in P_r(T)$ is the unique polynomial satisfying:
\begin{equation}\label{d-d}
(\nabla_{w,r,T} \cdot \bv, \varphi)_T = -(\bv_0, \nabla\varphi)_T + \langle \bv_b \cdot \mathbf{n}, \varphi\rangle_{\partial T}, \quad \forall \varphi \in P_r(T).
\end{equation}
By integration by parts, for $\bv_0 \in H(\text{div}, T)$, we have the equivalent form:
\begin{equation}\label{div2}
(\nabla_{w,r,T} \cdot \bv, \varphi)_T = (\nabla \cdot \bv_0, \varphi)_T + \langle (\bv_b - \bv_0) \cdot \mathbf{n}, \varphi\rangle_{\partial T}, \quad \forall \varphi \in P_r(T).
\end{equation}

\subsection{Discrete Weak Gradient}

For the pressure variable, we define a weak scalar function space $\mathcal{V}(T)$ as:
\begin{equation*} 
\mathcal{V}(T) = \{q = \{q_0, q_b\} : q_0 \in L^2(T), \ q_b \in L^2(\partial T)\}.
\end{equation*}
The  {weak gradient} of $q \in \mathcal{V}(T)$, denoted $\nabla_w q$, is a linear functional in the dual space of $[H^1(T)]^d$ such that for each $\bw \in [H^1(T)]^d$:
\begin{equation*}\label{wg}
(\nabla_w q, \bw)_T = -(q_0, \nabla \cdot \bw)_T + \langle q_b, \bw \cdot \mathbf{n}\rangle_{\partial T}.
\end{equation*}
The  {discrete weak gradient} $\nabla_{w,r, T} q \in [P_r(T)]^d$ is defined as the unique polynomial vector satisfying:
\begin{equation}\label{d-g}
(\nabla_{w,r, T} q, \bw)_T = -(q_0, \nabla \cdot \bw)_T + \langle q_b, \bw \cdot \mathbf{n}\rangle_{\partial T}, \quad \forall \bw \in [P_r(T)]^d.
\end{equation}
For $q_0 \in H^1(T)$, this is equivalently given by:
\begin{equation}\label{nabla2}
(\nabla_{w,r, T} q, \bw)_T = (\nabla q_0, \bw)_T + \langle q_b - q_0, \bw \cdot \mathbf{n}\rangle_{\partial T}, \quad \forall \bw \in [P_r(T)]^d.
\end{equation}
\section{A Least Squares Weak Galerkin Finite Element Scheme} 

Let $\mathcal{T}_h$ be a partition of the domain $\Omega$ with mesh size $h$, consisting of shape regular
   polygons or polyhedra. 
Denote by $\mathcal{E}_h$ the set of all edges (in 2D) or flat faces (in 3D) in $\mathcal{T}_h$, and let $\mathcal{E}_h^0 = \mathcal{E}_h \setminus \partial\Omega$ be the set of all interior edges or faces.

For any integer $k \ge 2$, we define the weak Galerkin finite element space for the velocity variable as follows:
\an{\label{W-h} \ad{ 
\mathcal{W}_h &= \{ \bv=\{\bv_0, \bv_b, \mathbf{v}_g\} 
   : \bv_0 \in [P_{k}(T)]^d, \ \bv_b \in [P_{k}(e)]^d, \\
   &\qquad \  \mathbf{v}_g \in [P_{k-1}(e)]^{d\times d},
     \ \forall T \in \mathcal{T}_h, \ \forall e \subset \partial T \}.
  } }
We emphasize that there is a single, uniquely defined value for $\bv_b$ on each edge $e \in \mathcal{E}_h^0$. For the pressure variable, the associated weak finite element space is given by:
\an{\label{V-h} \ad{ 
\mathcal{V}_h  & = \bigg\{ q=\{q_0, q_b\} : q_0 \in P_{k-1}(T) \cap L_0^2(\Omega), \\
   &\qquad \   q_b \in P_{k-1}(e), \ \forall T \in \mathcal{T}_h,  \ \forall e \subset \partial T 
     \bigg\}.
  } }
Let $\mathcal{W}_h^0$ denote the subspace of $\mathcal{W}_h$ consisting of discrete weak functions with vanishing boundary values; that is,
\begin{equation*}
\mathcal{W}_h^0 = \left\{ \bv=\{\bv_0, \bv_b, \mathbf{v}_g\} \in \mathcal{W}_h : \bv_b = \mathbf{0} \text{ on } \partial\Omega \right\}.
\end{equation*}

The discrete weak Laplacian $\Delta_{w,k-2}$ and the discrete weak divergence $(\nabla_{w,k-1}\cdot)$ on the space $\mathcal{W}_h$ are computed locally on each element $T$ using \eqref{dislap} and \eqref{d-d}, respectively. More precisely, they are defined by:
\begin{align*}
(\Delta_{w,k-2}\bv)|_T &= \Delta_{w,k-2, T} (\bv|_T), \quad \forall \bv \in \mathcal{W}_h,\\
(\nabla_{w,k-1}\cdot\bv)|_T &= \nabla_{w,k-1, T}\cdot (\bv|_T), \quad \forall \bv \in \mathcal{W}_h.
\end{align*}
For notational simplicity, we shall hereafter drop the subscripts $k-2$ and $k-1$ and denote the discrete weak Laplacian and discrete weak divergence simply by $\Delta_w$ and $\nabla_w\cdot$, respectively. 

Similarly, the discrete weak gradient $(\nabla_{w,k-2})$ on the space $\mathcal{V}_h$ is computed using \eqref{d-g} on each element $T$:
\begin{equation*}
(\nabla_{w,k-2} q)|_T = \nabla_{w,k-2, T} (q|_T), \quad \forall q \in \mathcal{V}_h.
\end{equation*}
Again, we will drop the subscript $k-2$ and denote the discrete weak gradient simply by $\nabla_w$.

Let $Q_{0}^k$, $Q_{0}^{k-1}$, and $Q_{0}^{k-2}$ denote the local $L^2$ projection operators from $L^2(T)$ onto $P_k(T)$, $P_{k-1}(T)$, and $P_{k-2}(T)$, respectively. For each face $e \in \mathcal{E}_h$, let $Q_{b}^k$ and $Q_{b}^{k-1}$ denote the $L^2$ projections from $L^2(e)$ onto $P_{k}(e)$ and $P_{k-1}(e)$, respectively.

We are now in a position to present the least squares weak Galerkin finite element scheme for the Stokes equations \eqref{moment}--\eqref{bc}. First, we introduce the following three bilinear forms: for $\bu, \bv, \bw\in \W_h$ and $p, q\in \V_h$,
\begin{align*}
s_1(\bv, \bw) &= \sum_{T\in \mathcal{T}_h}  h_T^{-1}\langle \bv_0-\bv_b, \bw_0-\bw_b\rangle_{\partial T}
   \\ &\qquad\qquad + h_T \langle \nabla \bv_0-\mathbf{v}_g, \nabla\bw_0-\mathbf{w}_g\rangle_{\partial T}, \\
s_2(p, q) &= \sum_{T\in \mathcal{T}_h} h_T\langle p_0-p_b, q_0-q_b\rangle_{\partial T}, \\
a((\bu,p),(\bv, q)) &= \sum_{T\in \mathcal{T}_h}  (-\Delta_w\bu+\nabla_w p, -\Delta_w\bv+\nabla_w q)_T + (\nabla_w\cdot\bu, \nabla_w\cdot\bv)_T.
\end{align*}

\begin{algorithm}
A numerical approximation for the Stokes problem \eqref{moment}--\eqref{bc} 
\ is obtained \ by seeking $\bu_h=\{\bu_0,\bu_b,$ $ \mathbf{u}_g\} \in \mathcal{W}_h$ and $p_h=\{p_0, p_b\} \in \mathcal{V}_h$ such that $\bu_b = Q_b^k\bg$ on $\partial\Omega$ and
\begin{equation}\label{lswg}
a((\bu_h,p_h),(\bv, q)) + s_1(\bu_h,\bv) + s_2(p_h,q) = \sum_{T\in \mathcal{T}_h} (\bbf, -\Delta_w\bv+\nabla_w q)_T,
\end{equation}
for all test functions $\bv=\{\bv_0,\bv_b, \mathbf{v}_g\} \in \mathcal{W}_h^0$ and $q=\{q_0, q_b\} \in \mathcal{V}_h$.
\end{algorithm}
 
\begin{theorem}\label{unique}
The Least Squares Weak Galerkin method \eqref{lswg} possesses a unique solution.  
\end{theorem}
\begin{proof}
It suffices to show that the solution to \eqref{lswg} is identically zero when the data vanishes (i.e., $\bbf=\mathbf{0}$ and $\bg=\mathbf{0}$). Assuming $\bbf=\mathbf{0}$ and $\bg=\mathbf{0}$, we choose the test functions $\bv=\bu_h$ and $q=p_h$ in \eqref{lswg} to obtain:
\begin{equation*}
a((\bu_h,p_h),(\bu_h, p_h)) + s_1(\bu_h,\bu_h) + s_2(p_h,p_h) = 0.
\end{equation*}
Due to the positivity of these bilinear forms, this implies that $-\Delta_w\bu_h+\nabla_w p_h = \mathbf{0}$ and $\nabla_w \cdot \bu_h = 0$ on each element $T\in \mathcal{T}_h$. Furthermore, the stabilizers dictate that $\bu_0=\bu_b$, $\nabla\bu_0=\mathbf{u}_g$, and $p_0=p_b$ on each $\partial T$. 

Using \eqref{dislap2} alongside the facts that $\bu_0=\bu_b$ and $\nabla\bu_0=\mathbf{u}_g$ on each $\partial T$, we have
\begin{equation*}
(\Delta_{w} \bu_h, \bw)_T = (\Delta \bu_0, \bw)_T, \quad \forall \bw \in [P_{k-2}(T)]^d,
\end{equation*}
which yields $\Delta_{w} \bu_h = \Delta \bu_0$. 

Similarly, using \eqref{div2} and $\bu_0=\bu_b$ on each $\partial T$ gives
\begin{equation*}
(\nabla_w \cdot \bu_h, \phi)_T = (\nabla \cdot \bu_0, \phi)_T, \quad \forall \phi \in P_{k-1}(T),
\end{equation*}
yielding $\nabla_w \cdot \bu_h = \nabla \cdot \bu_0$.

Applying \eqref{nabla2} and $p_0=p_b$ on each $\partial T$ yields
\begin{equation*}
(\nabla_w p_h, \bw)_T = (\nabla p_0, \bw)_T, \quad \forall \bw \in [P_{k-2}(T)]^d,
\end{equation*}
which implies $\nabla_w p_h = \nabla p_0$.

Because $\bu_0=\bu_b$ and $\nabla\bu_0=\mathbf{u}_g$ across all element boundaries, it follows that $\bu_0$ and its gradient are globally continuous, implying $\nabla\bu_0 \in [H(\text{div},\Omega)]^{d \times d}$. Likewise, $p_0=p_b$ implies that $p_0$ is continuous over the entire domain $\Omega$. Substituting $\Delta_{w} \bu_h = \Delta \bu_0$, $\nabla_w \cdot \bu_h = \nabla \cdot \bu_0$, and $\nabla_w p_h = \nabla p_0$ into the identities $-\Delta_w\bu_h+\nabla_w p_h=\mathbf{0}$ and $\nabla_w\cdot\bu_h=0$ yields the strong classical system:
\begin{align}
-\Delta\bu_0+\nabla p_0 &= \mathbf{0} \quad \text{in } \Omega, \label{md1}\\
\nabla\cdot\bu_0 &= 0 \quad \text{in } \Omega, \label{md2}\\
\bu_0 &= \mathbf{0} \quad \text{on } \partial\Omega. \label{md3}
\end{align}
Recall that $p_0 \in L_0^2(\Omega)$. Therefore, the continuous system \eqref{md1}--\eqref{md3} admits only the trivial solution $\bu_0 \equiv \mathbf{0}$ and $p_0 \equiv 0$ in $\Omega$. This forces the boundary terms $\bu_b \equiv \mathbf{0}$, $\bu_g \equiv \mathbf{0}$  and $p_b \equiv 0$, which ultimately guarantees that $\bu_h \equiv \mathbf{0}$ and $p_h \equiv 0$ throughout $\Omega$. 

This completes the proof. 
\end{proof}

We define the energy norm on the finite element space $\mathcal{W}_h^0 \times \mathcal{V}_h$ by
\begin{equation}
\3bar (\bu_h, p_h)\3bar^2 = a((\bu_h, p_h), (\bu_h, p_h)) + s_1(\bu_h, \bu_h) + s_2(p_h, p_h).
\end{equation}
Following the mathematical logic established in Theorem \ref{unique}, it is straightforward to verify that $\3bar \cdot \3bar$ satisfies all the properties of a norm on $\mathcal{W}_h^0 \times \mathcal{V}_h$, providing a rigorous framework for the subsequent optimal-order error analysis.

\section{Error Equations}

In this section, we derive the error equation that governs the relationship between the exact solutions $(\bu, p)$ and their LS-WG approximations $(\bu_h, p_h)$. Let $(\bu, p)$ be the exact solutions to the Stokes problem \eqref{moment}--\eqref{bc}, and let $\bu_h \in \mathcal{W}_h$ and $p_h\in \mathcal{V}_h$ be the solutions to the discrete problem \eqref{lswg}. We define the projection operators $Q_h^1 \bu$ and $Q_h^2 p$ into the weak spaces, and the corresponding error functions, as follows:
\begin{align*}
e_{\bu_h} &= \bu_h - Q_h^1 \bu = \{\bu_0 - Q_0^k \bu, \ \bu_b - Q_b^k \bu, \ \mathbf{u}_g - Q_b^{k-1}(\nabla \bu)\},  \\
e_{p_h} &= p_h - Q_h^2 p = \{p_0 - Q_0^{k-1}p, \ p_b - Q_b^{k-1}p\}.  
\end{align*}
 
\begin{lemma}[Commutative Properties]\label{Lemma:commute}
The projection operators $Q_h^1$ and $Q_h^2$ satisfy the following commutative properties:
\begin{align}
\Delta_w(Q_h^1 \bw) &= Q_0^{k-2}(\Delta \bw), \quad \forall \bw \in [H^2(\Omega)]^d, \label{eq:commute1} \\
\nabla_w\cdot(Q_h^1 \bw) &= Q_0^{k-1}(\nabla \cdot \bw), \quad \forall \bw \in [H^1(\Omega)]^d, \label{eq:commute2} \\
\nabla_w(Q_h^2 q) &= Q_0^{k-2}(\nabla q), \quad \forall q \in H^1(\Omega). \label{eq:commute3}
\end{align}
\end{lemma}

\begin{proof}
For any test function $\bv \in [P_{k-2}(T)]^d$, the definition of the discrete weak Laplacian \eqref{dislap} gives:
\begin{align*}
(\Delta_w (Q_h^1 \bw), \bv)_T &= (Q_0^k \bw, \Delta \bv)_T - \langle Q_b^k \bw, \nabla \bv\cdot \mathbf{n} \rangle_{\partial T} + \langle Q_b^{k-1}(\nabla \bw) \cdot \mathbf{n}, \bv \rangle_{\partial T} \\
&= (\bw, \Delta \bv)_T - \langle \bw, \nabla \bv \cdot \mathbf{n} \rangle_{\partial T} + \langle \nabla \bw \cdot \mathbf{n}, \bv \rangle_{\partial T} \\
&= (\Delta \bw, \bv)_T = (Q_0^{k-2} (\Delta \bw), \bv)_T.
\end{align*}
This confirms \eqref{eq:commute1}.

For any scalar function $\varphi \in P_{k-1}(T)$, the definition of the discrete weak divergence \eqref{d-d} implies:
\begin{align*}
(\nabla_w \cdot (Q_h^1 \bw), \varphi)_T &= -(Q_0^{k} \bw, \nabla \varphi)_T + \langle Q_b^k \bw \cdot \mathbf{n}, \varphi \rangle_{\partial T} \\
&= -(\bw, \nabla \varphi)_T + \langle \bw \cdot \mathbf{n}, \varphi \rangle_{\partial T} \\
&= (\nabla \cdot \bw, \varphi)_T = (Q_0^{k-1} (\nabla \cdot \bw), \varphi)_T.
\end{align*}
This confirms \eqref{eq:commute2}.

Finally, for any vector function $\bv \in [P_{k-2}(T)]^d$, the definition of the discrete weak gradient \eqref{d-g} yields:
\begin{align*}
(\nabla_w (Q_h^2 q), \bv)_T &= -(Q_0^{k-1} q, \nabla \cdot \bv)_T + \langle Q_b^{k-1} q, \bv \cdot \mathbf{n} \rangle_{\partial T} \\
&= -(q, \nabla \cdot \bv)_T + \langle q, \bv \cdot \mathbf{n} \rangle_{\partial T} \\
&= (\nabla q, \bv)_T = (Q_0^{k-2} (\nabla q), \bv)_T.
\end{align*}
This confirms \eqref{eq:commute3} and completes the proof.
\end{proof}

\begin{lemma}[Error Equation]\label{errorequa}
For any test functions $\bv \in \mathcal{W}_h^0$ and $q \in \mathcal{V}_h$, the error functions $e_{\bu_h}$ and $e_{p_h}$ satisfy the following identity:
\an{ \label{erroreqn}\ad{ &\quad \
a( (e_{\bu_h}, e_{p_h}), (\bv, q)) + s_1(e_{\bu_h}, \bv) + s_2(e_{p_h}, q)\\
 & = -s_1(Q_h^1 \bu, \bv) - s_2(Q_h^2 p, q). }}
\end{lemma}

\begin{proof}
Testing the exact momentum equation \eqref{moment} with $-\Delta_w \bv + \nabla_w q$ and the continuity equation \eqref{cont} with $\nabla_w \cdot \bv$ on each element $T \in \mathcal{T}_h$, and summing over all elements, we obtain:
\an{ \label{eq:err_proof1} \ad{ &\quad \ 
\sum_{T \in \mathcal{T}_h}   (- \Delta \bu + \nabla p, - \Delta_w \bv + \nabla_w q)_T + (\nabla\cdot \bu, \nabla_w\cdot \bv)_T \\
  &  = \sum_{T \in \mathcal{T}_h} (\bbf, - \Delta_w \bv + \nabla_w q)_T. }}
Note that by definition, $-\Delta_w \bv + \nabla_w q \in [P_{k-2}(T)]^d$ and $\nabla_w \cdot \bv \in P_{k-1}(T)$ on each element $T$. Therefore, we can apply the properties of the $L^2$ projections alongside the commutative properties \eqref{eq:commute1}--\eqref{eq:commute3} to deduce:
\a{ &\quad \ 
(- \Delta \bu, - \Delta_w \bv + \nabla_w q)_T \\
 &= (-Q_0^{k-2}(\Delta \bu), - \Delta_w \bv + \nabla_w q)_T = (-\Delta_w (Q_h^1 \bu), - \Delta_w \bv + \nabla_w q)_T, \\
 &\quad \ (\nabla p, - \Delta_w \bv + \nabla_w q)_T \\
 &= (Q_0^{k-2}(\nabla p), - \Delta_w \bv + \nabla_w q)_T = (\nabla_w (Q_h^2 p), - \Delta_w \bv + \nabla_w q)_T, \\
 &\quad \ (\nabla\cdot \bu, \nabla_w\cdot \bv)_T \\
  &= (Q_0^{k-1}(\nabla\cdot \bu), \nabla_w\cdot \bv)_T = (\nabla_w\cdot (Q_h^1 \bu), \nabla_w\cdot \bv)_T.
  }
Substituting these identities into \eqref{eq:err_proof1} recovers the bilinear form $a(\cdot,\cdot)$:
\begin{equation}\label{eq:err_proof2}
a((Q_h^1 \bu, Q_h^2 p), (\bv, q)) = \sum_{T \in \mathcal{T}_h} (\bbf, - \Delta_w \bv + \nabla_w q)_T.
\end{equation}
Subtracting \eqref{eq:err_proof2} from the discrete LS-WG numerical scheme \eqref{lswg} yields:
\begin{equation*}
a((\bu_h - Q_h^1 \bu, p_h - Q_h^2 p), (\bv, q)) + s_1(\bu_h, \bv) + s_2(p_h, q) = 0.
\end{equation*}
Using the linearity of the stabilization forms, we rewrite $s_1(\bu_h, \bv) = s_1(e_{\bu_h} + Q_h^1 \bu, \bv)$ and $s_2(p_h, q) = s_2(e_{p_h} + Q_h^2 p, q)$. Substituting these into the above equation directly yields the desired error equation:
\begin{equation*}
a((e_{\bu_h}, e_{p_h}), (\bv, q)) + s_1(e_{\bu_h}, \bv) + s_2(e_{p_h}, q) = -s_1(Q_h^1 \bu, \bv) - s_2(Q_h^2 p, q).
\end{equation*}
This completes the proof.
\end{proof}

\section{Error Estimates}\label{Section:ErrorEstimates}

In this section, we establish optimal-order error estimates for the LS-WG approximation in the energy norm. Throughout the analysis, we denote by $C$ a generic positive constant independent of the mesh parameter $h$.

\begin{lemma}[Approximation Properties]\label{lem:approx}
Let $\mathcal{T}_h$ be a shape-regular finite element partition of $\Omega$. For any $\bu \in [H^{k+1}(\Omega)]^d$ and $p\in H^{k}(\Omega)$, the following approximation estimates hold: 
\begin{align}
\sum_{T \in \mathcal{T}_h} \|\bu - Q_0^k \bu\|_{s,T}^2 &\le C h^{2(k+1-s)} \|\bu\|_{k+1}^2, \quad s=0, 1, 2, \label{error_uh} \\
\sum_{T \in \mathcal{T}_h} \|\Delta\bu - Q_0^{k-2} (\Delta\bu)\|_{T}^2 &\le C h^{2(k-1)} \|\bu\|_{k+1}^2, \label{error_uh2} \\
\sum_{T \in \mathcal{T}_h} \|p - Q_0^{k-1} p\|_{s,T}^2 &\le C h^{2(k-s)} \|p\|_{k}^2, \quad s=0, 1, \label{error_ph} \\
\sum_{T \in \mathcal{T}_h} \|\nabla p - Q_0^{k-2} (\nabla p)\|_{T}^2 &\le C h^{2(k-1)} \|p\|_{k}^2. \label{error_ph2}
\end{align}
Furthermore, for any $\phi \in H^1(T)$, the following trace inequality is valid:
\begin{equation}
\|\phi\|_{\partial T}^2 \le C \left( h_T^{-1} \|\phi\|_T^2 + h_T \|\nabla \phi\|_T^2 \right). \label{trace_H1} 
\end{equation}
\end{lemma}

\begin{theorem} \label{thm:discrete_error}
Let $(\bu, p) \in [H^{k+1}(\Omega)]^d \times H^k(\Omega)$ be the exact solutions of the Stokes problem \eqref{moment}--\eqref{bc}, and let $(\bu_h, p_h) \in \mathcal{W}_h \times \mathcal{V}_h$ be the LS-WG numerical solution defined by \eqref{lswg}. Then, there exists a constant $C > 0$ such that
\begin{equation}\label{main_estimate}
\3bar (e_{\bu_h}, e_{p_h}) \3bar  \le C h^{k} (\|\bu\|_{k+1} + \|p\|_{k}).
\end{equation}
\end{theorem}

\begin{proof}
By setting the test functions $\bv = e_{\bu_h}$ and $q = e_{p_h}$ in the error equation \eqref{erroreqn}, and recalling the definition of the energy norm $\3bar \cdot \3bar$, we have:
\begin{equation*}
\3bar (e_{\bu_h}, e_{p_h}) \3bar^2 = - s_1(Q_h^1 \bu, e_{\bu_h}) - s_2(Q_h^2 p, e_{p_h}).
\end{equation*}
 
\textbf{Bound for $s_1(Q_h^1 \bu, e_{\bu_h})$:} Applying the Cauchy-Schwarz inequality, the trace inequality \eqref{trace_H1}, and the approximation estimate \eqref{error_uh}, we obtain:
\begin{align}\label{s1} \ad{ &\quad \ 
s_1(Q_h^1 \bu, e_{\bu_h}) \\
 &\leq \left( \sum_{T \in \mathcal{T}_h} h_T^{-1} \|Q_0^k \bu - Q_b^k \bu\|_{\partial T}^2 + h_T \|\nabla Q_0^{k}\bu - Q_b^{k-1}(\nabla \bu)\|_{\partial T}^2 \right)^{\frac{1}{2}} \\
   &\qquad \3bar (e_{\bu_h}, e_{p_h})\3bar 
  \\ 
&\le C \left( \sum_{T \in \mathcal{T}_h} h_T^{-2} \|Q_0^k \bu - \bu\|_{T}^2 + \|Q_0^k \bu - \bu\|_{1, T}^2 \right. 
 \\
&\quad \left. + \|\nabla Q_0^{k}\bu - \nabla \bu\|_{T}^2 + h_T^2\|\nabla Q_0^{k}\bu - \nabla \bu\|_{1, T}^2 \right)^{\frac{1}{2}} \3bar (e_{\bu_h}, e_{p_h})\3bar 
 \\
&\le C h^{k} \|\bu\|_{k+1} \3bar (e_{\bu_h}, e_{p_h})\3bar. }
\end{align} 

\textbf{Bound for $s_2(Q_h^2 p, e_{p_h})$:} Similarly, applying the Cauchy-Schwarz inequality, the trace inequality \eqref{trace_H1} and the estimate \eqref{error_ph} yields:
\begin{align}\label{s2}
s_2(Q_h^2 p, e_{p_h}) &\leq \left( \sum_{T \in \mathcal{T}_h} h_T \|Q_0^{k-1} p - Q_b^{k-1} p\|_{\partial T}^2 \right)^{\frac{1}{2}} \3bar (e_{\bu_h}, e_{p_h})\3bar \nonumber \\
&\le C \left( \sum_{T \in \mathcal{T}_h} \|Q_0^{k-1} p - p\|_{T}^2 + h_T^2\|Q_0^{k-1} p - p\|_{1, T}^2 \right)^{\frac{1}{2}} \3bar (e_{\bu_h}, e_{p_h})\3bar \nonumber \\
&\leq C h^{k} \|p\|_{k} \3bar (e_{\bu_h}, e_{p_h})\3bar.
\end{align} 
Substituting \eqref{s1} and \eqref{s2} back into the energy norm identity and dividing by $\3bar (e_{\bu_h}, e_{p_h})\3bar$ completes the proof.
\end{proof}

\begin{theorem} \label{thm:global_error}
Let $(\bu, p) \in [H^{k+1}(\Omega)]^d \times H^{k}(\Omega)$ be the exact solutions of the Stokes problem \eqref{moment}--\eqref{bc}, and let $(\bu_h, p_h) \in \mathcal{W}_h \times \mathcal{V}_h$ be the LS-WG numerical solution defined by \eqref{lswg}. Then, there exists a constant $C > 0$ such that
\begin{equation}\label{main_estimate2}
\3bar (\bu-\bu_h, p - p_h) \3bar \le C h^{k-1} (\|\bu\|_{k+1} + \|p\|_{k}).
\end{equation}
\end{theorem}

\begin{proof} 
By the triangle inequality and the definition of the error functions, we have:
\begin{align*} &\quad \
\3bar (\bu-\bu_h, p- p_h) \3bar^2\\
  &\leq \3bar (\bu-Q_h^1\bu, p- Q_h^2p ) \3bar^2 + \3bar ( Q_h^1\bu-\bu_h, Q_h^2p-p_h ) \3bar^2 \\
&= \3bar (\bu-Q_h^1\bu, p- Q_h^2p ) \3bar^2 + \3bar (e_{\bu_h}, e_{p_h}) \3bar^2.
\end{align*}
Expanding the first term using the definition of the energy norm, and applying the commutative properties \eqref{eq:commute1}--\eqref{eq:commute3}, we obtain:
\begin{align*} &\quad \
\3bar (\bu-Q_h^1\bu, p- Q_h^2p ) \3bar^2 \\
  &= a((\bu-Q_h^1\bu, p- Q_h^2p ), (\bu-Q_h^1\bu, p- Q_h^2p )) \\
&\quad + s_1(\bu-Q_h^1\bu, \bu-Q_h^1\bu) + s_2 (p- Q_h^2p, p- Q_h^2p ) \\
&= \sum_{T\in \mathcal{T}_h} \|-\Delta \bu + Q_0^{k-2}(\Delta \bu) + \nabla p - Q_0^{k-2}(\nabla p) \|_T^2 \\
&\quad + \sum_{T\in \mathcal{T}_h} \|\nabla \cdot \bu - Q_0^{k-1}(\nabla\cdot \bu)\|_T^2 \\
&\quad + s_1(Q_h^1\bu, Q_h^1\bu) + s_2 (Q_h^2p, Q_h^2p).
\end{align*}
Note that $s_1(\bu, \cdot) = 0$ and $s_2(p, \cdot) = 0$ since the continuous solutions have no jumps across element boundaries. 
Using the approximation estimates \eqref{error_uh2} and \eqref{error_ph2}, along with the bounds established for the stabilizers in \eqref{s1} and \eqref{s2}, and the discrete error bound \eqref{main_estimate}, we have:
\begin{align*} &\quad \ 
\3bar (\bu-\bu_h, p- p_h) \3bar^2 \\
 &\leq C \sum_{T\in \mathcal{T}_h} \left( \|\Delta \bu - Q_0^{k-2}(\Delta \bu)\|_T^2 + \|\nabla p - Q_0^{k-2}(\nabla p) \|_T^2 \right) \\
&\quad + C \sum_{T\in \mathcal{T}_h} \|\nabla \cdot \bu - Q_0^{k-1}(\nabla\cdot \bu)\|_T^2 + C h^{2k} (\|\bu\|^2_{k+1} + \|p\|^2_{k}) \\ 
&\leq C h^{2k-2} \|\bu\|^2_{k+1} + C h^{2k-2} \|p\|^2_{k} + C h^{2k} (\|\bu\|^2_{k+1} + \|p\|^2_{k}) \\ 
&\leq C h^{2k-2} (\|\bu\|^2_{k+1} + \|p\|^2_{k}).
\end{align*}
Taking the square root of both sides yields the final estimate \eqref{main_estimate2}.
\end{proof}

\section{Numerical experiment}

In the first numerical test,  we solve the Stokes equations \eqref{moment}--\eqref{bc}, where 
   $\Omega=(0,1)\times(0,1)$, 
\an{\label{sol1} \ad{ \ \b u&=\p{8(1-2y)(y-y^2)(x-x^2)^2\\
        -8 (1-2x)(x-x^2)(y-y^2)^2}, \ \\\ p&=4 (1-2x)(x-x^2)(1-2y)(y-y^2) .  }  }
The computation is done on the triangular grids shown in Figure \ref{f-g1} and on the
   non-convex pentagonal  
  grids shown in Figure \ref{f-g2}, by 
  the weak Galerkin $P_k$-$P_{k-1}^2$ finite elements defined in
   \eqref{W-h} and \eqref{V-h}, $k=2,3$ and $4$.
The results are listed in Tables \ref{t1}-\ref{t3}, 
    where we can see that the optimal orders of convergence 
  are achieved roughly.  
  In these tables, $G_i$ denotes the $i$-th grid, shown in
    Figure \ref{f-g1} or Figure \ref{f-g2} .

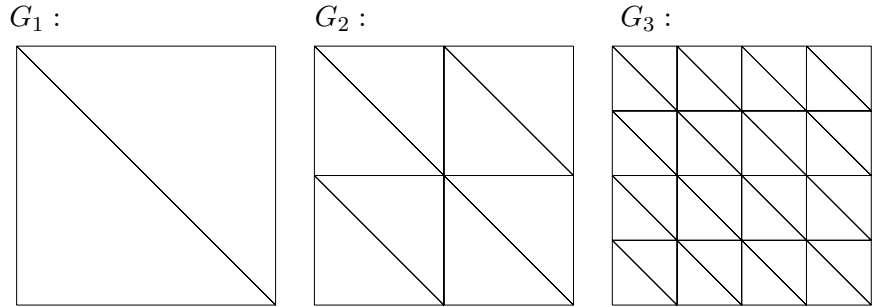
\begin{figure}[H]
 \begin{center}\setlength\unitlength{1.0pt}
\begin{picture}(340,110)(0,0) 
 \put(0,102){$G_1:$}  \put(115,102){$G_2:$} \put(230,102){$G_3:$} 
  
  \put(0,-10){\includegraphics[width=330pt]{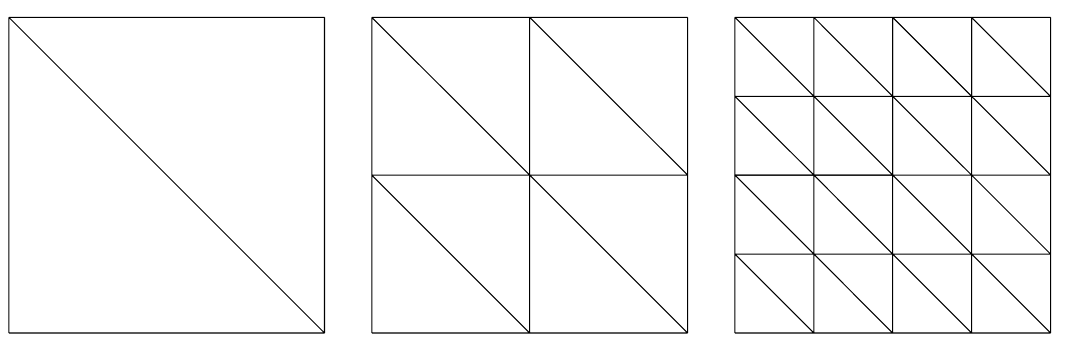}}   
 \end{picture}\end{center}
\caption{The triangular grids for \eqref{sol1}, used in Tables \ref{t1}--\ref{t3}. }\label{f-g1}
\end{figure}

\begin{table}[H]
  \centering  \renewcommand{\arraystretch}{1.1}
  \caption{Error profile by the $P_2$-$P_1$ WG element  for computing \eqref{sol1}. }
  \label{t1}
\begin{tabular}{c|cc|cc|cc}
\hline
  $G_i$ & \quad $\| u-u_h\|_{0}$ & $O(h^r)$ & \  $\|\Delta_w( u-u_h)\|_0 $& $O(h^r)$ 
    &$ \| p-p_h\|_{0}$ & $O(h^r)$  \\ \hline
    &  \multicolumn{6}{c}{On triangular grids (Figure \ref{f-g1}) }   \\
\hline  
 1&    0.603E-02 & --- &    0.463E-01 & --- &    0.236E-01 & ---  \\
 2&    0.308E-02 &  1.0&    0.170E-01 &  1.4&    0.147E-01 &  0.7 \\
 3&    0.814E-03 &  1.9&    0.889E-01 &  0.0&    0.910E-02 &  0.7 \\
 4&    0.193E-03 &  2.1&    0.486E-01 &  0.9&    0.258E-02 &  1.8 \\
 5&    0.490E-04 &  2.0&    0.265E-01 &  0.9&    0.653E-03 &  2.0 \\
 6&    0.127E-04 &  2.0&    0.137E-01 &  1.0&    0.163E-03 &  2.0 \\
\hline 
    &  \multicolumn{6}{c}{On pentagonal grids (Figure \ref{f-g2}) }   \\
\hline  
 1&    0.256E-01 & --- &    0.259E+00 & --- &    0.598E-01 & ---  \\
 2&    0.663E-02 &  1.9&    0.576E-01 &  2.2&    0.301E-01 &  1.0 \\
 3&    0.115E-02 &  2.5&    0.791E-01 &  0.0&    0.108E-01 &  1.5 \\
 4&    0.217E-03 &  2.4&    0.401E-01 &  1.0&    0.334E-02 &  1.7 \\
 5&    0.529E-04 &  2.0&    0.227E-01 &  0.8&    0.884E-03 &  1.9 \\
 6&    0.143E-04 &  1.9&    0.123E-01 &  0.9&    0.224E-03 &  2.0 \\
\hline 
    \end{tabular}%
\end{table}%

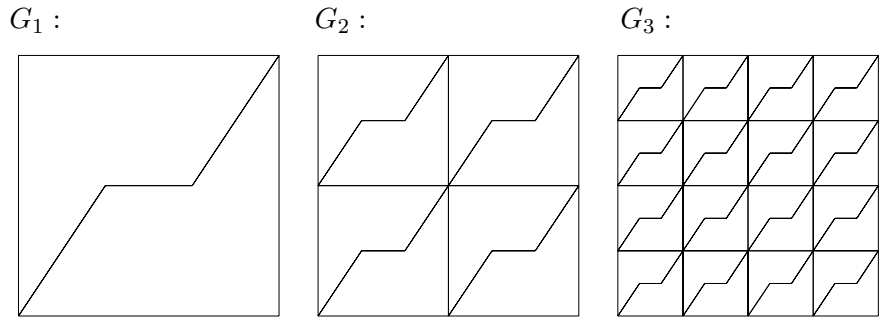
\begin{figure}[H]
 \begin{center}\setlength\unitlength{1.0pt}
\begin{picture}(340,110)(0,0) 
 \put(0,102){$G_1:$}  \put(115,102){$G_2:$} \put(230,102){$G_3:$} 
  
  \put(0,-10){\includegraphics[width=330pt]{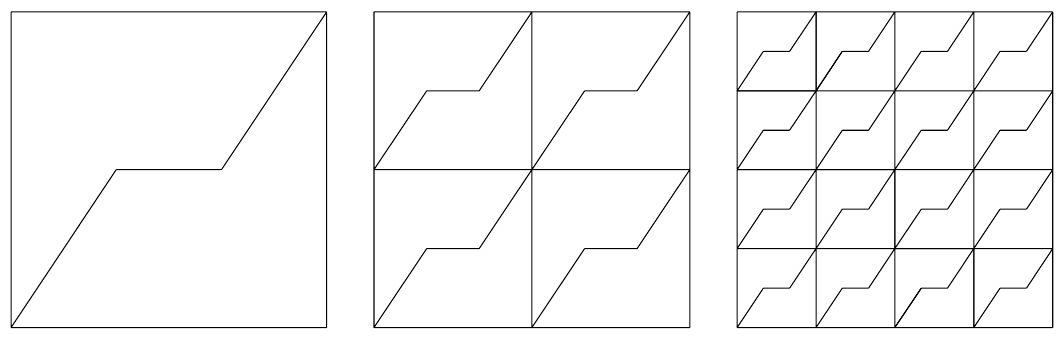}}   
 \end{picture}\end{center}
\caption{The pentagonal grids for \eqref{sol1}, used in Tables \ref{t1}--\ref{t3}. }\label{f-g2}
\end{figure}

\begin{table}[H]
  \centering  \renewcommand{\arraystretch}{1.1}
  \caption{Error profile by the $P_3$-$P_2$ WG element  for computing \eqref{sol1}. }
  \label{t2}
\begin{tabular}{c|cc|cc|cc}
\hline
  $G_i$ & \quad $\| u-u_h\|_{0}$ & $O(h^r)$ & \  $\|\Delta_w( u-u_h)\|_0 $& $O(h^r)$ 
    &$ \| p-p_h\|_{0}$ & $O(h^r)$  \\ \hline
    &  \multicolumn{6}{c}{On triangular grids (Figure \ref{f-g1}) }   \\
\hline  

 1&    0.472E-02 & --- &    0.942E-02 & --- &    0.297E-01 & ---  \\
 2&    0.602E-03 &  3.0&    0.142E+00 &  0.0&    0.109E-01 &  1.4 \\
 3&    0.775E-04 &  3.0&    0.460E-01 &  1.6&    0.177E-02 &  2.6 \\
 4&    0.522E-05 &  3.9&    0.115E-01 &  2.0&    0.216E-03 &  3.0 \\
 5&    0.296E-06 &  4.1&    0.259E-02 &  2.1&    0.256E-04 &  3.1 \\\hline 
    &  \multicolumn{6}{c}{On pentagonal grids (Figure \ref{f-g2}) }   \\
\hline  
 1&    0.246E-01 & --- &    0.866E+00 & --- &    0.179E+00 & ---  \\
 2&    0.229E-02 &  3.4&    0.161E+00 &  2.4&    0.210E-01 &  3.1 \\
 3&    0.233E-03 &  3.3&    0.399E-01 &  2.0&    0.322E-02 &  2.7 \\
 4&    0.162E-04 &  3.8&    0.961E-02 &  2.1&    0.409E-03 &  3.0 \\
 5&    0.105E-05 &  4.0&    0.222E-02 &  2.1&    0.532E-04 &  2.9 \\
\hline 
    \end{tabular}%
\end{table}%

\begin{table}[H]
  \centering  \renewcommand{\arraystretch}{1.1}
  \caption{Error profile by the $P_4$-$P_3$ WG element  for computing \eqref{sol1}. }
  \label{t3}
\begin{tabular}{c|cc|cc|cc}
\hline
  $G_i$ & \quad $\| u-u_h\|_{0}$ & $O(h^r)$ & \  $\|\Delta_w( u-u_h)\|_0 $& $O(h^r)$ 
    &$ \| p-p_h\|_{0}$ & $O(h^r)$  \\ \hline
    &  \multicolumn{6}{c}{On triangular grids (Figure \ref{f-g1}) }   \\
\hline   
 1&    0.312E-02 & --- &    0.343E+00 & --- &    0.327E-01 & ---  \\
 2&    0.265E-03 &  3.6&    0.145E+00 &  1.2&    0.601E-02 &  2.4 \\
 3&    0.116E-04 &  4.5&    0.138E-01 &  3.4&    0.320E-03 &  4.2 \\
 4&    0.363E-06 &  5.0&    0.142E-02 &  3.3&    0.183E-04 &  4.1 \\
 \hline 
    &  \multicolumn{6}{c}{On pentagonal grids (Figure \ref{f-g2}) }   \\
\hline  
 1&    0.298E-01 & --- &    0.140E+01 & --- &    0.214E+00 & ---  \\
 2&    0.186E-02 &  4.0&    0.183E+00 &  2.9&    0.142E-01 &  3.9 \\
 3&    0.741E-04 &  4.6&    0.252E-01 &  2.9&    0.110E-02 &  3.7 \\
 4&    0.250E-05 &  4.9&    0.320E-02 &  3.0&    0.715E-04 &  3.9 \\
\hline 
    \end{tabular}%
\end{table}%

For a 3D numerical test,  we solve the Stokes equations \eqref{moment}--\eqref{bc}, where 
   $\Omega=(0,1)^3$, 
\an{\label{sol2} \ad{ \ \b u&=\p{- 2^{10}(x-x^2)^2  (y-y^2)^2 (1-2z)(z-z^2)\\ 
    \ 2^{10}(x-x^2)^2  (y-y^2)^2 (1-2z)(z-z^2)\\  
        2^{10} (2 x^3 - 2 y^3 - 3 x^2 + 3 y^2 + x - y) (z-z^2)^2}, \\
         p&=z^3 -\frac 14 .  }  }
We also solve the problem with a smoother solution:
\an{\label{s3} \ad{ \ \b u&=\p{e^y \\ 
     e ^ z\\   e^x }, &
         p&=y -\frac 12 .  }  }

\begin{figure}[H]
 \begin{center}\setlength\unitlength{1.0pt}
\begin{picture}(340,92)(0,0) 
 \put(0,85){$G_1:$}  \put(115,85){$G_2:$} \put(230,85){$G_3:$} 
  
  \put(0,-10){\includegraphics[width=330pt]{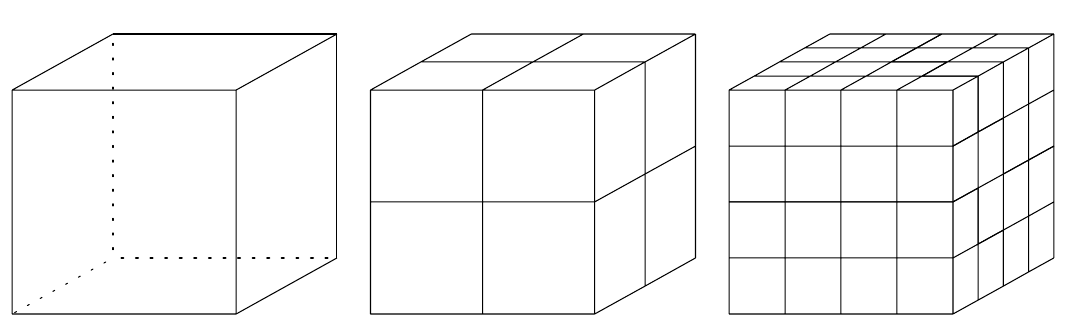}}   
 \end{picture}\end{center}
\caption{The cuboid grids for solving \eqref{sol2} and \eqref{s3}. }\label{f-g3}
\end{figure}
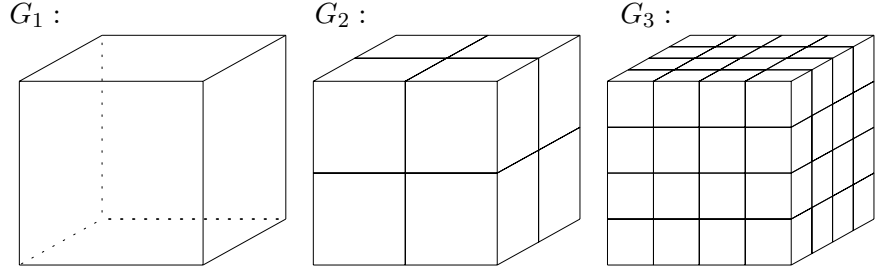

\begin{figure}[H]
 \begin{center}\setlength\unitlength{1.0pt}
\begin{picture}(340,92)(0,0) 
 \put(0,85){$G_1:$}  \put(115,85){$G_2:$} \put(230,85){$G_3:$} 
  
  \put(0,-10){\includegraphics[width=330pt]{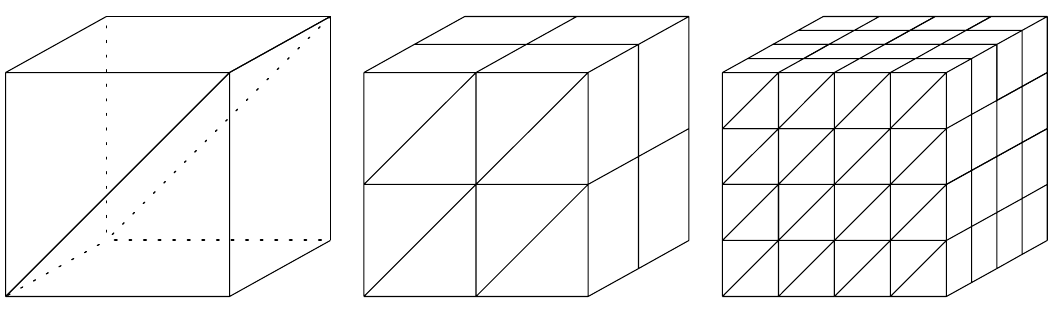}}   
 \end{picture}\end{center}
\caption{The wedge-shape grids for solving \eqref{sol2} and \eqref{s3}. }\label{f-g4}
\end{figure}
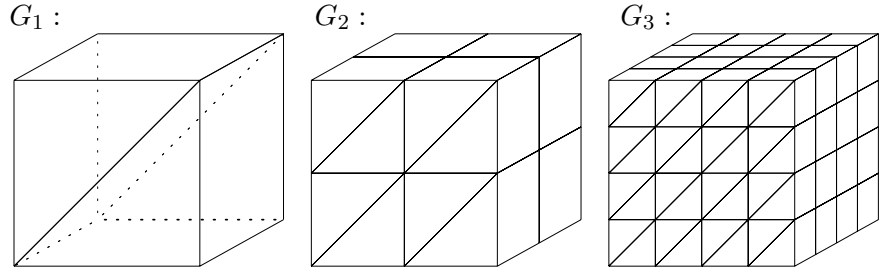

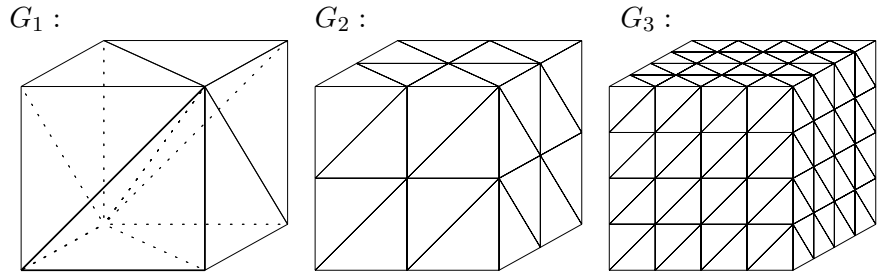
\begin{figure}[H]
 \begin{center}\setlength\unitlength{1.0pt}
\begin{picture}(340,92)(0,0) 
 \put(0,85){$G_1:$}  \put(115,85){$G_2:$} \put(230,85){$G_3:$} 
  
  \put(0,-10){\includegraphics[width=330pt]{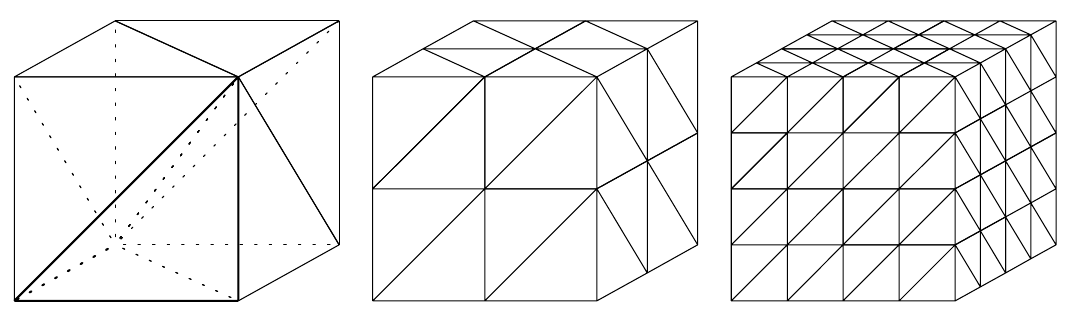}}   
 \end{picture}\end{center}
\caption{The tetrahedral grids for solving \eqref{sol2} and \eqref{s3}. }\label{f-g5}
\end{figure}

The computation is done on the cuboid grids shown in Figure \ref{f-g3},
  on the wedge-shape grids  shown in Figure \ref{f-g4}, and on the
  tetrahedral grids shown in Figure \ref{f-g5}, by 
  the weak Galerkin $P_k$-$P_{k-1}^2$ finite elements defined in
   \eqref{W-h} and \eqref{V-h}, $k=2,3$ and $4$.
The results are listed in Tables \ref{t4}-\ref{t6}, 
    where we can see barely that the optimal orders of convergence 
  are achieved,
 due to large condition numbers and the computer accuracy.

\begin{table}[H]
  \centering  \renewcommand{\arraystretch}{1.1}
  \caption{Error profile by the 3D $P_2$-$P_1$ WG finite element. }
  \label{t4}
\begin{tabular}{c|cc|cc|cc}
\hline
  $G_i$ & \quad $\| u-u_h\|_{0}$ & $O(h^r)$ & \  $\|\Delta_w( u-u_h)\|_0 $& $O(h^r)$ 
    &$ \| p-p_h\|_{0}$ & $O(h^r)$  \\ \hline
    &  \multicolumn{6}{c}{On cuboid grids (Figure \ref{f-g3}) for \eqref{s3} }   \\
\hline   
 1 &   0.195E-01 &0.00 &   0.147E+00 &0.00 &   0.733E-01 &0.00 \\
 2 &   0.284E-02 &2.78 &   0.679E-01 &1.11 &   0.242E-01 &1.60 \\
 3 &   0.395E-03 &2.85 &   0.171E-01 &1.99 &   0.523E-02 &2.21 \\
 4 &   0.521E-04 &2.92 &   0.416E-02 &2.04 &   0.121E-02 &2.11 \\
 5 &   0.684E-05 &2.93 &   0.101E-02 &2.04 &   0.291E-03 &2.06 \\
 \hline 
    &  \multicolumn{6}{c}{On wedge grids (Figure \ref{f-g4}) for \eqref{s3} }   \\
\hline  
 1 &   0.120E-01 &0.00 &   0.237E+00 &0.00 &   0.507E-01 &0.00 \\
 2 &   0.191E-02 &2.65 &   0.115E+00 &1.05 &   0.228E-01 &1.15 \\
 3 &   0.269E-03 &2.83 &   0.285E-01 &2.00 &   0.469E-02 &2.28 \\
 4 &   0.365E-04 &2.88 &   0.101E-01 &1.50 &   0.104E-02 &2.17 \\ 
 \hline 
    &  \multicolumn{6}{c}{On tetrahedral grids (Figure \ref{f-g5}) for \eqref{sol2} }   \\
\hline   
 1 &   0.116E+00 &0.00 &   0.243E+01 &0.00 &   0.726E+00 &0.00 \\
 2 &   0.836E-01 &0.47 &   0.255E+01 &0.00 &   0.492E+00 &0.56 \\
 3 &   0.239E-01 &1.81 &   0.207E+01 &0.30 &   0.179E+00 &1.46 \\
 4 &   0.420E-02 &2.51 &   0.943E+00 &1.14 &   0.508E-01 &1.82 \\
\hline   
    &  \multicolumn{6}{c}{On cuboid grids (Figure \ref{f-g3}) for \eqref{sol2} }   \\
\hline   
 1 &   0.690E+00 &0.00 &   0.592E-11 &0.00 &   0.182E+02 &0.00 \\
 2 &   0.337E+00 &1.03 &   0.175E+02 &0.00 &   0.738E+00 &4.62 \\
 3 &   0.665E-01 &2.34 &   0.278E+01 &2.65 &   0.185E+00 &1.99 \\
 4 &   0.936E-02 &2.83 &   0.506E+00 &2.46 &   0.495E-01 &1.90 \\
 5 &   0.186E-02 &2.33 &   0.222E+00 &1.19 &   0.159E-01 &1.63 \\
\hline   
    &  \multicolumn{6}{c}{On wedge grids (Figure \ref{f-g4}) for \eqref{sol2} }   \\
\hline   
 1 &   0.538E+00 &0.00 &   0.241E+02 &0.00 &   0.201E+01 &0.00 \\
 2 &   0.196E+00 &1.45 &   0.865E+01 &1.48 &   0.575E+00 &1.81 \\
 3 &   0.431E-01 &2.19 &   0.300E+01 &1.53 &   0.186E+00 &1.63 \\
 4 &   0.627E-02 &2.78 &   0.133E+01 &1.18 &   0.459E-01 &2.02 \\
\hline   
    \end{tabular}%
\end{table}%

\begin{table}[H]
  \centering  \renewcommand{\arraystretch}{1.1}
  \caption{Error profile by the 3D $P_3$-$P_2$ WG finite element. }
  \label{t5}
\begin{tabular}{c|cc|cc|cc}
\hline
  $G_i$ & \quad $\| u-u_h\|_{0}$ & $O(h^r)$ & \  $\|\Delta_w( u-u_h)\|_0 $& $O(h^r)$ 
    &$ \| p-p_h\|_{0}$ & $O(h^r)$  \\ \hline
    &  \multicolumn{6}{c}{On cuboid grids (Figure \ref{f-g3}) for \eqref{s3} }   \\
\hline   
 1 &   0.159E-02 &0.00 &   0.640E-01 &0.00 &   0.285E-01 &0.00 \\
 2 &   0.195E-03 &3.03 &   0.729E-02 &3.13 &   0.182E-02 &3.97 \\
 3 &   0.138E-04 &3.82 &   0.106E-02 &2.78 &   0.140E-03 &3.70 \\
 4 &   0.903E-06 &3.94 &   0.140E-03 &2.93 &   0.993E-05 &3.82 \\
 \hline 
    &  \multicolumn{6}{c}{On wedge grids (Figure \ref{f-g4}) for \eqref{s3} }   \\
\hline  
 1 &   0.106E-02 &0.00 &   0.661E-01 &0.00 &   0.178E-01 &0.00 \\
 2 &   0.147E-03 &2.85 &   0.110E-01 &2.59 &   0.161E-02 &3.47 \\
 3 &   0.102E-04 &3.85 &   0.196E-02 &2.49 &   0.119E-03 &3.76 \\
 \hline 
    &  \multicolumn{6}{c}{On cuboid grids (Figure \ref{f-g3}) for \eqref{sol2} }   \\
\hline   
 1 &   0.103E+01 &0.00 &   0.950E+02 &0.00 &   0.808E+01 &0.00 \\
 2 &   0.254E+00 &2.01 &   0.167E+02 &2.50 &   0.156E+01 &2.37 \\
 3 &   0.291E-01 &3.13 &   0.204E+01 &3.04 &   0.113E+00 &3.78 \\
 4 &   0.383E-02 &2.93 &   0.411E+00 &2.31 &   0.941E-02 &3.59 \\
\hline    
    &  \multicolumn{6}{c}{On wedge grids (Figure \ref{f-g4}) for \eqref{sol2} }   \\
\hline   
 1 &   0.578E+00 &0.00 &   0.546E+02 &0.00 &   0.415E+01 &0.00 \\
 2 &   0.117E+00 &2.31 &   0.108E+02 &2.33 &   0.372E+00 &3.48 \\
 3 &   0.100E-01 &3.55 &   0.352E+01 &1.62 &   0.434E-01 &3.10 \\
\hline   
    \end{tabular}%
\end{table}%

\begin{table}[H]
  \centering  \renewcommand{\arraystretch}{1.1}
  \caption{Error profile by the 3D $P_4$-$P_3$ WG finite element. }
  \label{t6}
\begin{tabular}{c|cc|cc|cc}
\hline
  $G_i$ & \quad $\| u-u_h\|_{0}$ & $O(h^r)$ & \  $\|\Delta_w( u-u_h)\|_0 $& $O(h^r)$ 
    &$ \| p-p_h\|_{0}$ & $O(h^r)$  \\ \hline 
    &  \multicolumn{6}{c}{On cuboid grids (Figure \ref{f-g3}) for \eqref{sol2} }   \\
\hline   
 1 &   0.480E+01 &0.00 &   0.421E+03 &0.00 &   0.686E+02 &0.00 \\
 2 &   0.393E+00 &3.61 &   0.462E+02 &3.19 &   0.126E+01 &5.76 \\
 3 &   0.162E-01 &4.60 &   0.542E+01 &3.09 &   0.600E-01 &4.40 \\
\hline    
    &  \multicolumn{6}{c}{On wedge grids (Figure \ref{f-g4}) for \eqref{sol2} }   \\
\hline   
 1 &   0.719E+00 &0.00 &   0.106E+03 &0.00 &   0.739E+01 &0.00 \\
 2 &   0.961E-01 &2.90 &   0.569E+02 &0.90 &   0.212E+00 &5.12 \\
 3 &   0.492E-02 &4.29 &   0.807E+01 &2.82 &   0.113E-01 &4.23 \\
\hline   
    &  \multicolumn{6}{c}{On tetrahedral grids (Figure \ref{f-g5}) for \eqref{sol2} }   \\
\hline   
 1 &   0.881E-01 &0.00 &   0.744E+01 &0.00 &   0.508E+00 &0.00 \\
 2 &   0.919E-02 &3.26 &   0.302E+01 &1.30 &   0.122E+00 &2.06 \\
 3 &   0.349E-02 &1.40 &   0.432E+00 &2.81 &   0.978E-02 &3.64 \\
\hline   
    \end{tabular}%
\end{table}%

\end{document}